\documentclass[a4paper,12pt,english,reqno]{amsart}
\usepackage{graphicx, varioref, amscd ,color, bm, stmaryrd, mathrsfs, dsfont, esint}
\usepackage[colorlinks=false]{hyperref}

\usepackage{tikz} 
\usepackage{pgflibraryarrows} 
\usepackage{placeins} 

\newcommand{\R}{\mathbb{R}}

\theoremstyle{plain}

\theoremstyle{plain}
\newtheorem{theorem}{Theorem}[section]
\newtheorem{lemma}{Lemma}[section]

\theoremstyle{definition}

\newtheorem{remark}[lemma]{Remark}

\numberwithin{equation}{section}
\allowdisplaybreaks[1]

\title[A Logarithmically damped wave system]
{ The Global Existence and Energy Decay of a Partially Logarithmically Damped Wave System}

\author[S. A. Messaoudi]{S. A. Messaoudi}
\address[Salim A. Messaoudi]
{\newline Department of Mathematics,
\newline College of Sciences,
\newline University of Sharjah,
\newline P.O. Box 27272,
\newline Sharjah, UAE}
\email[]{smessaoudi@sharjah.ac.ae}

\author[S. Ulusoy]{S. Ulusoy}
\address[S\"uleyman Ulusoy]
{\newline Department of Mathematics and Physics,
\newline College of Arts and Sciences,
\newline American University of Ras Al Khaimah,
\newline P.O. Box 10021,
\newline Ras Al Khaimah, UAE}
\email[]{suleyman.ulusoy@aurak.ac.ae}

\date{\today}



\begin{document}

\begin{abstract}
We investigate the global existence and energy decay properties of a weakly coupled evolution system consisting of two wave equations in a bounded domain. The system exhibits a unique structural asymmetry: a nonlinear logarithmic damping term acts exclusively as a friction-like feedback mechanism on the first wave equation, while the second equation remains completely undamped and relies solely on linear cross-coupling to transfer and dissipate its energy. 

We establish two primary objectives. First, by utilizing the Faedo-Galerkin approximation scheme in conjunction with the structural monotonicity properties of the non-Lipschitz logarithmic nonlinearity, we prove a rigorous global existence result for strong solutions. Second, we provide a rigorous mathematical proof demonstrating the stability and asymptotic energy decay of the solutions under suitable constraints on the cross-coupling parameter. These results provide new insights into the stabilization mechanisms of coupled hyperbolic systems under degenerate nonlinear damping.
\end{abstract}

\maketitle
\allowdisplaybreaks


\section{Introdcution}

In recent decades, the study of the well-posedness and asymptotic behavior of solutions to evolution equations has attracted a great deal of attention, due to their fundamental role in understanding stability and long-term dynamics. In particular, coupled systems of wave equations modeling various physical mechanisms, such as thermoelasticity, electromagnetic interaction, or mechanical vibration, have been extensively investigated. The nature of the dissipation (whether it is linear, nonlinear, localized, or memory-driven) and the number of controls play a crucial role in determining the rate of energy decay.

In this regard, we mention the work of Almeida and M. L. Santos~\cite{Almeida2011} who considered, in a bounded domain, the following partially controlled system:
\begin{align*}
    u_{tt} - \mu \Delta u + \int_{0}^{\infty} g(s)\Delta u(t - s)\,ds + \alpha v &= 0, \quad \text{in } \Omega \times (0, \infty), \\
    v_{tt} - \mu \Delta v + \alpha u &= 0, \quad \text{in } \Omega \times (0, \infty),
\end{align*}
together with initial and Dirichlet-boundary conditions, where $\alpha$ is a positive and small parameter and the kernel $g$ decays exponentially. They proved that the energy decays polynomially but does not decay exponentially. We also mention the important work of Jin et al.~\cite{Jin2014}, who studied a coupled system of two abstract evolution equations with finite memory. Precisely, they looked into the following system:
\begin{align*}
    u_{tt} + Au + \int_{0}^{t} g(s)Au(t - s)\,ds + \alpha u + \beta Bv &= f(u), \quad t > 0, \\
    v_{tt} + Av + \beta Bu &= 0, \quad t > 0,
\end{align*}
where $A$ and $B$ are operators, $\alpha \geq 0$, $\beta > 0$ are two constants, and $f$ is a nonlinear source. Under appropriate assumptions on $g$, the operators $A$, $B$, and $f$, they proved polynomial stability with a decay rate of order $1/t$. Guesmia et al.~\cite{Guesmia2024} studied an abstract coupled system involving one viscoelastic component interacting with a purely elastic equation. They established some general decay rates for the total energy of the system and showed that asymptotic behavior strongly depends on the decay behavior of the relaxation function and the coupling parameter.

Recently, Fadel et al.~\cite{Fadel2026} considered the following system:
\begin{align*}
    u_{tt} - \mu \Delta u + \int_{0}^{t} g(s)\Delta u(t - s)\,ds + \alpha v &= 0, \quad \text{in } \Omega \times (0, \infty), \\
    v_{tt} - \mu \Delta v + \alpha u &= 0, \quad \text{in } \Omega \times (0, \infty),
\end{align*}
where the relaxation function satisfies, for an increasing and convex function $G$ near the origin and a nonincreasing function $\xi$, the relation:
\[
g'(t) \leq -\xi(t)G(g(t)).
\]
They established, in details, an existence and uniqueness result, then proved a general and explicit result for energy decay and provided some numerical illustration to validate the theoretical findings. Several extensions and related contributions in this direction can be found in~\cite{Akil2021, Guesmia2015, Oliveira2023}.

Besides the viscoelastic systems, purely elastic coupled wave models have also been investigated. For instance, Lobato et al.~\cite{Lobato2014} looked into the weakly coupled system:
\begin{align*}
    u_{tt} - \mu \Delta u + u_t + \alpha v &= 0, \quad \text{in } \Omega \times (0, \infty), \\
    v_{tt} - \mu \Delta v + \alpha u &= 0, \quad \text{in } \Omega \times (0, \infty),
\end{align*}
where $\Omega$ is a bounded domain and $\alpha$ a positive constant. In their work, the authors showed that the energy decays polynomially, with the rate being optimal. They also provided numerical illustrations in the one-dimensional case. Similar results on stability and decay rates had earlier been established for systems with boundary damping by Komornik et al.~\cite{Komornik1997} and Aassila~\cite{Aassila1999, Aassila2001}.

Contrary to the huge literature on classical damping mechanisms, the treatment of logarithmic-type damping remains comparatively small, particularly within the setting of coupled systems. Typical examples of such dampings, mentioned in~\cite{Li2025}, include
\[
\text{sgn}(u_t)\ln(1 + |u_t|), \quad |u_t|^q u_t \ln(1 + u_t^2), \quad q \geq 2,
\]
which all lack linearity, homogeneity, and coercivity. This, of course, creates substantial analytical difficulties for establishing both well-posedness and stability.

Motivated by the work of Li et al.~\cite{Li2025} on the wave equation, many  works extended these ideas and results to other problems, such as Timoshenko systems, swelling porous media and plate problems. We refer to~\cite{AlMahdi2026_swelling, AlMahdi2026_piezo, Messaoudi2026_timo, Messaoudi2025_timo, Messaoudi2026_weak, Messaoudi2026_visco} for more information.

Inspired by the aforementioned works, particularly the recent developments in logarithmic dampings and weakly coupled multi-component systems, we consider the following system:
 
\begin{equation}\label{a1}
\left\{
\begin{split}
u_{tt} - \Delta u &+ \alpha u_t \ln(1+u_t^2) + \beta v = 0, \quad (x, t) \in \Omega \times (0, \infty)\\
v_{tt} - \Delta v &+ \beta u = 0, \quad (x, t) \in \Omega \times (0, \infty)\\
u(x, t) &= 0 = v(x,t), \quad (x, t) \in \partial \Omega \times (0, T)\\
u(0,x) &= u_0(x), u_t(x,0) = u_1(x), \quad x \in \Omega\\
 v(x, 0) &= v_0(x), v_t(x,0) = v_1(x), \quad x \in \Omega,
\end{split}
\right.
\end{equation}
where $\Omega$ is a bounded domain of $\R^d$ $(d \geq 1) $ with a smooth boundary $\partial \Omega$, $\alpha, \beta > 0, $ are constants and we assume that the initial data belong to suitable  spaces. \textbf{To the best of our knowledge}, this is the first work on systems of wave equations associated with only one logarithmic damping.\\

In this paper, we want to achieve two goals. Namely, 

\begin{enumerate}

\item Prove a rigorous global existence result for solutions to the problem \eqref{a1}, using the Faedo-Galerkin method for which monotonicity of the nonlinear term plays a role. To the best of our knowledge, no such existence result is previously available for this problem.\\

\item Provide a rigorous proof of the stability of solutions. 

\end{enumerate}

The rest of the paper is organized as follows. In Section 2, some preliminary material that is essential for the analysis is presented. In Section 3, the existence result is presented and proved. Section 4 contains the discussion of the stability of the energy of the solution.

 \section{Preliminaries}
 
 Here, we introduce the mathematical tools and preliminary results used in the proofs of the main theorems. We consider the standard Lebesgue space $L^2(\Omega) $ and the Sobolev spaces $H_0^1(\Omega)$ and $H^2(\Omega)$, equipped with usual norms and inner products. Throughout, $c$ and $C$ denote generic positive constants. \\

\begin{remark}
  Throughout the paper we assume that the constant $\beta$ is such that
  \begin{equation}\label{eq:cond-beta}
0 < \beta < \frac{1}{C_p},  
  \end{equation}
where  $C_p$ is the Poincar\'{e} constant.
  
 \end{remark}

 \begin{remark}  (See \cite{Li-Zhang-Zhang}.) The following inequalities are used for the existence and the decay results.\\
 
 \begin{itemize}
 
 \item For $|s| \leq 1, m \geq 1$, and $0  < n \leq 2$, there exist positive constants $C_p$ and $C_q$ such that 
 
 \begin{equation}\label{b1}
 C_p|s|^{1+m} \leq |s| \ln(1+|s|) \leq C_q|s|^n.
 \end{equation}
 
 \item For $|s| \geq 1$, $0 < k < 1$, and $\rho > 0$, there exist positive constants $C_k$ and $C_p$ such that 
 
 \begin{equation}\label{b2}
 C_k|s|^{k} \leq |s| \ln(1+|s|) \leq C_{\rho}|s|^{1+\rho}, \quad \text{for} \, \, |s| \geq 1, \rho > 0, 0 < k \leq 1.
 \end{equation}
 
 \end{itemize}
 \end{remark}

 \section{Existence of Weak Solutions}
 
 In this section, we examine the well-posedness of the problem \eqref{a1}. This is done by combining some energy estimates and the Faedo-Galerkin method.
 
 \begin{theorem}\label{thm:existence}
 Assume \eqref{eq:cond-beta} holds and that  $u_0, v_0 \in H^2(\Omega) \cap H_0^1(\Omega)$  and $u_1, v_1 \in H_0^1(\Omega)$. Then, the problem \eqref{a1} has a unique weak solution $(u,v)$ satisfying for every $T > 0$,
 $$u, v \in L^{\infty}([0,T), H^2(\Omega) \cap H_0^1(\Omega)),$$
 $$ u_t, v_t \in L^2((0, T), H_0^1(\Omega)) \cap H_0^1((0, T) \times \Omega),$$
 $$ u_{tt}, v_{tt} \in L^{\infty}((0, T), L^2(\Omega)).$$
 \end{theorem}
 
 \begin{proof}
 
 ~\\
 
 \noindent \textbf{Step 1.} \emph{Faedo-Galerkin approximation.}
 We take $\{\phi_j\}_{j=1}^\infty$ to be the eigenfunctions of the Laplacian operator subject to Dirichlet boundary conditions, satisfying $-\Delta \phi_j = \lambda_j \phi_j$. It is well-known that $\{\phi_j\}_{j=1}^\infty$ forms an orthonormal basis for $H_0^1(\Omega), \, H^2(\Omega) \cap H_0^1(\Omega)$, and $L^2(\Omega)$. Let $V_N$ be the vector space spanned by the first $N$ functions $\phi_1, \phi_2, \cdots, \phi_N$. We will seek for an approximate solution $(u^N, v^N)$ of the form
 $$ u^N(x,t) = \sum_{j=1}^N a_j(t)\phi_j(x), \quad v^N(x,t) = \sum_{j=1}^N a_j(t)\phi_j(x), \quad t \in (0, T),$$
 which satisfies, for $1 \leq j \leq N,$

 \begin{equation}\label{eq-w1}
 \left\{
 \begin{split}
 \int_{\Omega} u_{tt}^N \phi_j \, dx &+ \int_{\Omega} \nabla u^N \cdot \nabla \phi_j \, dx + \alpha \int_{\Omega} u_t^N \ln(1+(u_t^N)^2) \phi_j \, dx \\
 & \qquad \qquad + \beta \int_{\Omega} v^N \phi_j \, dx = 0, \\
 \int_{\Omega} v_{tt}^N \phi_j \, dx &+ \int_{\Omega} \nabla v^N \cdot \nabla \phi_j \, dx  + \beta \int_{\Omega} u^N \phi_j \, dx = 0, 
 \end{split}
 \right. 
 \end{equation}
with 
 \begin{equation}\label{eq-w2}
 \left\{
 \begin{split}
 u^N(0) &= u_0^N = \sum_{j=1}^N <u_0, \phi_j>\phi_j \to u_0,\\
 v^N(0) &= v_0^N = \sum_{j=1}^N <v_0, \phi_j>\phi_j \to u_0, \,\, \text{in} \,\, H^2(\Omega) \cap H_0^1(\Omega), \\
 \end{split}
 \right. 
 \end{equation}
 
  \begin{equation}\label{eq-w3}
 \left\{
 \begin{split}
 u_t^N(0) &= u_1^N = \sum_{j=1}^N <u_1, \phi_j>\phi_j \to u_0,\\
 v_t^N(0) &= v_1^N = \sum_{j=1}^N <v_1, \phi_j>\phi_j \to u_0, \,\, \text{in} \,\, H_0^1(\Omega). \\
 \end{split}
 \right. 
 \end{equation}
 Standard ODE theory guarantees the existence of a unique local solution $(u^N, v^N)$ of \eqref{eq-w1}-\eqref{eq-w3} defined on $[0, t_N), \, \, 0 < t_N \leq T$. In the next steps we obtain a priori estimates that will allow us to extend the solution to the entire interval. \\

  \noindent \textbf{Step 2.} \\
   \noindent \textbf{A priori estimate 1.}
 Multiplying the first equation in \eqref{eq-w1} by $a_j'(t)$ and summing over $j=1$ to $N$, one gets for all $0<t \leq t_N$,
 \begin{equation}\label{eq-w4}
 \begin{split}
 \frac{d}{dt} \frac{1}{2} \left( ||u_t^N(t)||_2^2 + ||\nabla u^N(t)||_2^2 \right) &+ \alpha \int_{\Omega} (u_t^N)^2 \ln(1+(u_t^N)^2) \, dx \\
 & + \beta \int_{\Omega} v^N(t) u_t^N(t) \, dx = 0.
 \end{split}
 \end{equation}
  Multiplying the second equation in \eqref{eq-w1} by $b_j'(t)$ and summing over $j=1$ to $N$, one gets for all $0<t \leq t_N$,
  \begin{equation}\label{eq-w5}
 \frac{d}{dt} \frac{1}{2} \left( ||v_t^N(t)||_2^2 + ||\nabla v^N(t)||_2^2 \right) + \beta \int_{\Omega} u^N(t) v_t^N(t) \, dx = 0.
 \end{equation}
 Combining \eqref{eq-w4} and \eqref{eq-w5} implies 
 \begin{equation}\label{eq-w6}
 \begin{split}
 &\frac{d}{dt} \frac{1}{2}\left( ||u_t^N||_2^2+||v_t^N||_2^2 +  ||\nabla u^N||_2^2  +  ||\nabla v^N||_2^2 + 2\beta \int_{\Omega} u^N(t) v^N(t) \, dx\right) \\
 & \qquad \qquad \qquad \qquad  = -\alpha \int_{\Omega} (u_t^N)^2 \ln(1+(u_t^N)^2) \, dx,
 \end{split}
 \end{equation}
which yields 
\begin{equation}\label{eq-w7}
\begin{split}
 &\frac{d}{dt} \frac{1}{2}\left( ||u_t^N||_2^2+||v_t^N||_2^2 +  ||\nabla u^N||_2^2  +  ||\nabla v^N||_2^2 + 2\beta \int_{\Omega} u^N(t) v^N(t) \, dx\right)\\
& \qquad \qquad \qquad \qquad   \leq 0.
\end{split}
 \end{equation}
 Integrating \eqref{eq-w7} over $(0, t)$ gives 
 \begin{equation}\label{eq-w8}
 E_{1,N}(t) \leq E_{1,N}(0),
 \end{equation}
 where 
$$ E_{1,N}(t) =||u_t^N||_2^2+||v_t^N||_2^2+||\nabla u^N||_2^2+||\nabla v^N||_2^2+2\beta \int_{\Omega} u^N(t)v^N(t) \, dx,$$
$$  E_{1,N}(0) =||u_1^N||_2^2+||v_1^N||_2^2+||\nabla u_0^N||_2^2+||\nabla v_0^N||_2^2+2\beta \int_{\Omega} u_0^N v_0^N \, dx.$$
By Young's and Poincar\'{e}'s inequalities, for $\varepsilon_1>0$, we have
\begin{equation}\label{eq-w9}
 2\int_{\Omega} u^N(t)v^N(t) \, dx \geq -\varepsilon_1C_p||\nabla u^N(t)||_2^2-\frac{C_p}{\varepsilon_1}||\nabla v^N(t)||_2^2
 \end{equation}
 and 
\begin{equation}\label{eq-w10}
 2\beta \int_{\Omega} u_0^N(t)v_0^N(t) \, dx \leq \beta C_p||\nabla u_0^N||_2^2+\beta C_p||\nabla v_0^N||_2^2.
 \end{equation} 
 This implies that 
 \begin{equation}\label{eq-w11}
 E_{1,N}(0) \leq ||u_1^N||_2^2+||v_1^N||_2^2+(1+\beta C_p)||\nabla u_0^N||_2^2 +(1+\beta C_p)||\nabla v_0^N||_2^2.
 \end{equation}
The sequences $(u_0^N)_{N \in \mathbb{N}}, (u_1^N)_{N \in \mathbb{N}}, (v_0^N)_{N \in \mathbb{N}}, (v_1^N)_{N \in \mathbb{N}}$  converge, so there is  a positive constant $C_1$, independent of $N$, such that
  \begin{equation}\label{eq-w11}
 E_{1,N}(0) \leq C_1.
 \end{equation}
 On the other hand, we also have 
   \begin{equation}\label{eq-w12}
   \begin{split}
 E_{1,N}(t) &\geq ||u_t^N||_2^2+ ||v_t^N||_2^2+\left[1-\beta \varepsilon_1C_p\right]||\nabla u^N||_2^2\\
& \qquad \qquad  +\left[1-\frac{\beta C_p}{\varepsilon_1}\right] ||\nabla v^N||_2^2.
 \end{split}
\end{equation}
We now choose $\varepsilon_1$ such that $\beta C_p < \varepsilon_1 < \frac{1}{\beta C_p}$. Hence, there exists a constant $C > 0$, such that 
\begin{equation}\label{eq-w13}
 E_{1,N}(t) \geq C \left[ ||u_t^N||_2^2+||\nabla u^N||_2^2 +||v_t^N||_2^2 +||\nabla v^N||_2^2  \right].
 \end{equation}
 $E_{1,N}(t) \leq E_{1,N}(0)$ and $E_{1,N}(0) \leq C_1$ imply that
 \begin{equation}\label{eq-w14}
 ||u_t^N(t)||_2^2+||\nabla u^N(t)||_2^2 +||v_t^N(t)||_2^2 +||\nabla v^N(t)||_2^2 \leq \frac{C_1}{C}, \quad \forall t \leq t_N \leq T,
 \end{equation}
 where $\frac{C_1}{C}$ is a constant  independent of $t$ and $N$.\\
 
 \noindent \textbf{A priori estimate 2.} In the first equation of \eqref{eq-w1}, we replace $\phi_j$ by $-\Delta \phi_j$ and then multiply it by $a_j'(t)$ and in the second equation of \eqref{eq-w1}, we replace $\phi_j$ by $-\Delta \phi_j$ and then multiply it by $b_j'(t)$. By summming the resulting equations from $j=1$ to $j=N$, we arrive at
 \begin{equation*}
 \begin{split}
 \frac{d}{dt}\frac{1}{2} \Bigg[ ||\nabla u_t^N||_2^2+||\nabla v_t^N||_2^2 &+ ||\Delta u^N||_2^2+||\Delta v^N||_2^2 
 \\ 
 \qquad &+2\beta \int_{\Omega}\nabla u^N \cdot \nabla v^N \, dx \Bigg] \leq 0.
 \end{split}
 \end{equation*}
 
 Applying Young's and Poincar\'{e}'s inequalities, we have, similar to \eqref{eq-w14},
 
\begin{equation}\label{eq-r4}  
  ||\nabla u_t^N||_2^2 + ||\nabla v_t^N||_2^2 + ||\Delta u^N||_2^2 + ||\Delta v^N||_2^2 \leq \frac{K_1}{C}, \quad \forall t \leq t_N \leq T,
\end{equation}
where $\frac{K_1}{C}$  is a constant independent of $t$ and $N$.\\

\noindent \textbf{A priori estimate 3.} Multiplying the first equation in \eqref{eq-w1} by $a_j''(t)$ and summing over $j=1$ to $N$, one gets for all $0<t \leq t_N$,
 \begin{equation}\label{ss1}
 \begin{split}
  \int_{\Omega} \, |u_{tt}^N| \, dx &= \int_{\Omega} \Delta u^N u_{tt}^N \, dx\\ 
  &\quad -\beta  \int_{\Omega} v^N u_{tt}^N \, dx \\
 & \qquad - \alpha \int u_t^N u_{tt}^N \ln(1+(u_t^N)^2) \, dx.
 \end{split}
 \end{equation}
For the first term on the right side of \eqref{ss1}, one has, by Young's inequality,
 \begin{equation}\label{ss2}
 \int_{\Omega} \Delta u^N u_{tt}^N \, dx \leq \varepsilon \int_{\Omega} |u_{tt}^N|^2 \, dx + \frac{1}{4\varepsilon} \int_{\Omega} |\Delta u^N|^2 \, dx.
 \end{equation}
 
 For the second term on the right side of \eqref{ss1}, one has, by Young's inequality,
 \begin{equation}\label{ss3}
 -\beta \int_{\Omega} v^N u_{tt}^N \, dx \leq \varepsilon \int_{\Omega} |u_{tt}^N|^2 \, dx + \frac{\beta^2}{4\varepsilon} \int_{\Omega} |\Delta u^N|^2 \, dx.
 \end{equation}
 
  For the last term on the right side of \eqref{ss1}, one has, by Young's inequality,
 \begin{equation}\label{ss4}
 \begin{split}
& -\alpha \int_{\Omega} u_t^N u_{tt}^N \ln(1+(u_t^N)^2) \, dx \\
& \qquad \leq \varepsilon \int_{\Omega} |u_{tt}^N|^2 \, dx + \frac{\alpha^2}{4\varepsilon} \int_{\Omega} (u_t^N)^2 \ln^2(1+(u_t^N)^2) \, dx.
 \end{split}
 \end{equation}
We now deal with the last term in \eqref{ss4}. Define the following partition of $\Omega$.
$$\Omega^+ := \{x \in \Omega: |u_t(x)|> 1\} \quad \text{and} \quad \Omega^- := \{x \in \Omega: |u_t(x)| \leq 1\}.$$
 
 On $\Omega^-$, one has
\begin{equation}\label{ss5} 
 \int_{\Omega^-} (u_t^N)^2 \ln^2(1+(u_t^N)^2) \, dx \leq |\Omega| \ln^2(2) < \infty.
\end{equation} 
 
By \eqref{b1} with $s = (u_t^N)^2$ and choosing $\rho > 0$ such that $2(1+\rho) \leq \frac{2d}{d-2}$ if $d \geq 3$ and $\rho > 0$ if $d= 1, 2$ and using the Sobolev embedding theorem and \eqref{eq-r4}, one obtains
 \begin{equation}\label{ss6} 
 \begin{split}
 \int_{\Omega^+} (u_t^N)^2 \ln^2(1+(u_t^N)^2) \, dx &\leq   c \int_{\Omega^+} |u_t^N|^{2(1+\rho)} \, dx \\
 &\qquad  \leq c ||\nabla u_t^N||_2^{\rho + 1} < \infty.
\end{split}
\end{equation} 
 
Combining \eqref{ss5} and \eqref{ss6}, we deduce that 
$$ \int_{\Omega} (u_t^N)^2 \ln(1+(u_t^N)^2) \, dx < \infty.$$  
 
Taking $\varepsilon$ small enough in \eqref{ss1}-\eqref{ss4},  we deduce that

\begin{equation}\label{ss7}
 ||u_{tt}^N||_2^2 \leq L_3.
\end{equation} 
 Repeating similar calculations with the second equation in  \eqref{eq-w1}
 \begin{equation}\label{ss8}
 ||v_{tt}^N||_2^2 \leq L_4,
\end{equation} 
 where $L_3, L_4$ are constants  independent of $t$ and $N$.

 From \eqref{eq-w14}, \eqref{eq-r4}, \eqref{ss7} and \eqref{ss8} we conclude that\\
 
 $(u^N)_{N \in \mathbb{N}}$ and $(v^N)_{N \in \mathbb{N}}$ are bounded in $L^{\infty}((0, T); H^2(\Omega) \cap H_0^1(\Omega))$,\\
 
  $(u_t^N)_{N \in \mathbb{N}}$ and $(v_t^N)_{N \in \mathbb{N}}$ are bounded in $L^{\infty}((0, T); H_0^1(\Omega))$,\\
 
  $(u_{tt}^N)_{N \in \mathbb{N}}$ and $(v_{tt}^N)_{N \in \mathbb{N}}$ are bounded in $L^{\infty}((0, T); L^2(\Omega))$.\\

 \noindent \textbf{Step 3.} From the previous step, there is a subsequences of $(u^N)_{N \in \mathbb{N}}$ and $(v^N)_{N \in \mathbb{N}}$ (still denoted by $(u^N)_{N \in \mathbb{N}}$ and $(v^N)_{N \in \mathbb{N}}$),  and two functions $u, v : \Omega \times [0, T) \to \mathbb{R}$ such that \\
 
 $u^N \overset{\ast}{\rightharpoonup} u$ and $v^N \overset{\ast}{\rightharpoonup} v$ in $L^{\infty}((0, T); H^2(\Omega) \cap H_0^1(\cap)),$\\
 
  $u_t^N \overset{\ast}{\rightharpoonup} u_t$ and $v_t^N \overset{\ast}{\rightharpoonup} v_t$ in $L^{\infty}((0, T); H_0^1(\cap)),$\\
 
  $u_{tt}^N \overset{\ast}{\rightharpoonup} u_{tt}$ and $v_{tt}^N \overset{\ast}{\rightharpoonup} v_{tt}$ in $L^{\infty}((0, T); L^2(\Omega)).$\\
 
 Thanks to the above convergences, taking $N \to \infty $ in \eqref{eq-w1}, one obtains
 \begin{equation*}
 \begin{split}
 \int_{\Omega} u_{tt} \phi_j \, dx &+ \int_{\Omega} \nabla u \cdot \nabla \phi_j \, dx 
 + \beta \int_{\Omega} v \phi_j \, dx \\ 
 & \qquad+ \alpha \int_{\Omega} u_t \ln(1+u_t^2) \phi_j \, dx = 0, \, 1 \leq j \leq N,
 \end{split}
 \end{equation*}
  \begin{equation*}
 \int_{\Omega} v_{tt} \phi_j \, dx + \int_{\Omega} \nabla v \cdot \nabla \phi_j \, dx 
 + \beta \int_{\Omega} v \phi_j \, dx  = 0, \, 1 \leq j \leq N.
 \end{equation*}

 This remains true $\forall \phi \in H_0^1(\Omega) \cap H^2(\Omega)$. Consequently, integration by parts yields
 \begin{equation*}
 \int_{\Omega} \left[ u_{tt} - \Delta u + \alpha u_t \ln(1+u_t^2) + \beta v\right] \, \phi \, dx = 0,
 \end{equation*}
  \begin{equation*}
 \int_{\Omega} \left[ v_{tt} - \Delta v + \beta u\right] \, \phi \, dx = 0.
 \end{equation*}
 
 Therefore, $(u, v)$ solves the equation \eqref{a1}.\\
 
  \noindent \textbf{Step 4. Initial Conditions:} First, by Lions' Lemma \cite{Lions} and since 
  
   $$u^N \overset{\ast}{\rightharpoonup} u \, \, \, \text{in} \, \, \,  L^{\infty}((0, T); H^2(\Omega) \cap H_0^1(\cap)),$$
 
  $$u_t^N \overset{\ast}{\rightharpoonup} u_t \, \, \, \text{in} \, \, \,  L^{\infty}((0, T); H_0^1(\cap)),$$
 
 we deduce, up to a subsequence, that
 
 $$u^N \to u \, \, \, \text{ in} \, \, \,  C([0,T]; H_0^1(\Omega)).$$
 
 Therefore, $u^N(\cdot, 0)$ is defined and 
 
 $$u^N(\cdot, 0) \to u(\cdot, 0) \, \, \,  \text{in} \, \, \,  H_0^1(\Omega).$$
 
 But,
 
 $$u^N(\cdot, 0) = u_0^N \to u_0 \, \, \,  \text{in} H^2(\Omega) \cap H_0^1(\cap)).$$
 
 Then, $u(\cdot, 0) = u_0$. Similarly, $v(\cdot, 0) = v_0$.\\
 
 Also, noting that 
 
 $$u_{tt}^N \overset{\ast}{\rightharpoonup} u_{tt}  \, \, \, \text{and} \, \, \,  v_{tt}^N \overset{\ast}{\rightharpoonup} v_{tt} \, \, \, \text{in} \, \, \, L^{\infty}((0, T); L^2(\Omega)),$$ 
 
\noindent we obtain $u_t, v_t \in C([0, T); L^2(\Omega))$, as in Lions \cite{Lions}, and thus $u_t(\cdot, 0)$ has a meaning. Furthermore,\\
 
\noindent $u_t^N(\cdot, 0) \to u_t(\cdot, 0)$ and $v_t^N(\cdot, 0) \to v_t(\cdot, 0)$ in $L^2(\Omega)$. But, $u_t^N(\cdot, 0) = u_1^N \to u_1$ and $v_t^N(\cdot, 0) = v_1^N \to v_1$ in $H_0^1(\Omega)$. Hence, $u_t(x, 0) = u_1(x)$ and $v_t(x, 0) = v_1(x)$.\\
 
  \noindent \textbf{Uniqueness:} Assume that problem \eqref{a1} has 2 strong solutions, $(u, v)$ and $(\bar{u}, \bar{v})$. Then, $(\mathbf{u}, \mathbf{v}) = (u-\bar{u}, v-\bar{v})$ satisfies 
 
 \begin{equation}\label{un-1}
 \mathbf{u}_{tt} - \Delta \mathbf{u} + u_t \ln(1+u_t^2) - \bar{u}_t \ln(1+ \bar{u}_t^2) + \beta \mathbf{v} = 0, \quad \text{in} \quad \Omega \times (0, T),
 \end{equation}
  \begin{equation}\label{un-2}
 \mathbf{v}_{tt} - \Delta \mathbf{v} + \beta \mathbf{u} = 0, \quad \text{in} \quad \Omega \times (0, T),
 \end{equation}
 \begin{equation}\label{un-3}
 \mathbf{u}(x, t) = \mathbf{v}(x, t) = 0, \quad \text{on} \quad \partial \Omega \times (0, T),
 \end{equation}
  \begin{equation}\label{un-4}
 \mathbf{u}(x, 0) = \mathbf{u}_t (x, 0) = 0, \, \,  \mathbf{v}(x, t) = \mathbf{v}_t(x, t) = 0, \quad \text{in} \quad  \Omega.
 \end{equation}
Multiplying \eqref{un-1} by $\mathbf{u}_t$ and \eqref{un-2} by $\mathbf{v}_t$, integrating over $\Omega \times (0, t)$, and adding the resulting two equations, we find 
 \begin{equation}\label{un-5}
 \begin{split}
 \frac{1}{2}\int_{\Omega} \mathbf{u}_t^2 \, dx +& \frac{1}{2}\int_{\Omega} |\nabla \mathbf{u}|^2 \, dx + \frac{1}{2} \int_{\Omega} |\mathbf{v}_t|^2 \, dx + \frac{1}{2}\int_{\Omega} |\nabla \mathbf{v}|^2 \, dx\\
 & + \int_0^t \int_{\Omega} \left( A(u_t) - A(v_t) \right) \left( u_t - v_t\right) \, dx \, ds \\
 & \qquad + 2 \beta \int_{\Omega} \mathbf{v} \mathbf{u} \, dx = 0,
 \end{split}
 \end{equation}
where $A(f) := f \ln(1+f^2)$. We note that the monotonicity of $A$ implies that  
 $$ \int_0^t \int_{\Omega} \Big( A(u_t) - A(v_t) \Big) \Big( u_t - v_t\Big) \, dx \, ds \geq 0.$$

 
 
 
 Similarly to \eqref{eq-w12}, we arrive at
 
 $$ ||\mathbf{u}_t(t)||_2^2 + ||\nabla \mathbf{u}(t)||_2^2 +||\mathbf{v}_t(t)||_2^2 +   ||\nabla \mathbf{v}(t)||_2^2 = 0.$$

 Hence, $\mathbf{u}_t(\cdot, t) = \mathbf{v}_t(\cdot, t) = 0$ and  $\nabla \mathbf{u}(\cdot, t) = \nabla \mathbf{v}(\cdot, t) = 0$ for all $t \in (0, T)$. This implies that $\mathbf{u} = \mathbf{v} = 0$ in $\Omega \times (0, T)$ given that $\mathbf{u} = \mathbf{v} = 0$ on $\partial \Omega \times (0, T)$. This completes the proof of the uniqueness.

 \end{proof}

 \section{Stability}

 In this section, we establish the decay result. We first define the first-order and second-order energies as follows:
 
 \begin{equation}\label{b3}
 E_1(t) := \frac{1}{2} \left[ ||u_t||_2^2+||\nabla u||_2^2 +  ||v_t||_2^2+||\nabla v||_2^2\right] + \beta \int_{\Omega} uv \, dx.
 \end{equation}
 Directly differentiating \eqref{b3}, using \eqref{a1}, we get
 \begin{equation}\label{b4}
 E_1'(t) = -\alpha \int_{\Omega} u_t^2 \ln(1+u_t^2) \, dx \leq 0.
 \end{equation}
 
 Now, we define the second order energy as
 
  \begin{equation}\label{b5}
 E_2(t) := \frac{1}{2} \left[ ||u_{tt}||_2^2+||\nabla u_t||_2^2 +  ||v_{tt}||_2^2+||\nabla v_t||_2^2\right] + \beta \int_{\Omega} u_t v_t \, dx.
 \end{equation}
 
  Directly differentiating \eqref{b5}, using \eqref{a1}, we get
 
  \begin{equation}\label{b6}
 E_2'(t) = -\alpha \int_{\Omega} u_{tt}^2 \ln(1+u_t^2) \, dx - \alpha   \int_{\Omega} \frac{u_{t}^2}{1+u_t^2} u_{tt}^2  \, dx    \leq 0.
 \end{equation}

\begin{lemma}\label{lemma-d1}
Let $(u,v) $ be the solution of \eqref{a1}. Then, there exists a constant $m> 0$ such that, for all $t \geq 0$,
\begin{equation}\label{d2}
E_1(t) \geq m \left[ ||u_t||_2^2+||\nabla u||_2^2 +  ||v_t||_2^2+||\nabla v||_2^2\right],
\end{equation}
and
\begin{equation}\label{d3}
E_2(t) \geq m  \left[ ||u_{tt}||_2^2+||\nabla u_t||_2^2 +  ||v_{tt}||_2^2+||\nabla v_t||_2^2\right].
\end{equation}

\end{lemma}

\begin{proof}
Applying Cauchy-Schwarz, Poincar\'{e} and Young's inequalities, for $\varepsilon_1 > 0$, one obtains
$$ E_1(t) \geq \frac{1}{2} \left[ ||u_t||_2^2+\Big(1-\beta C_p\varepsilon_1 \Big||\nabla u||_2^2 +  ||v_t||_2^2+(1-\frac{\beta C_p}{\varepsilon_1} )||\nabla v||_2^2\right] ,$$
for all $t \geq 0$, where $C_p$ is the Poincar\'{e} constant. We select $\varepsilon_1 $ such that $\beta C_p \leq \varepsilon_1 \leq \frac{1}{\beta C_p}$. Consequently, \eqref{d2} is derived by choosing $m = \frac{1}{2}\min \{1-\beta C_p\varepsilon_1, 1-\frac{\beta C_p}{\varepsilon_1}\}$. Using a similar reasoning, \eqref{d3} is obtained with the same constant.

\end{proof}

\begin{lemma}\label{lemma-d4}
Let $(u,v)$ be the solution of \eqref{a1} and let
\begin{equation}\label{d5}
\psi(t) := \int_{\Omega} u u_t \, dx +  \int_{\Omega} v v_t \, dx. 
\end{equation}
Then, 
\begin{equation}\label{d6}
\begin{split}
\psi'(t) &\leq  ||u_t||_2^2 + ||v_t||_2^2 + (\beta C_p-1)||\nabla u||_2^2 + (\beta C_p-1)||\nabla v||_2^2 \\
 & \quad + 2\varepsilon C_pE_1(t) + C(\varepsilon)\left(-E_1'(t)\right)\\
 & \qquad + C E_1^{1/2}(0) \left( -E_1'(t)\right)^{{(\gamma-1)}/\gamma}, 
\end{split}
\end{equation}
where $C_p$ is the Poincar\'{e} constant and $\gamma$ is chosen such that $2 < \gamma \leq \frac{2d}{d-2}$ if $d \geq 3$ and $\gamma > 2$ if $n=1, 2$.

\end{lemma}

\begin{proof}
Routine calculations yield
$$\psi'(t) = ||u_t||_2^2+||v_t||_2^2-||\nabla u||_2^2-||\nabla v||_2^2-2\beta\int_{\Omega} uv \, dx - \alpha \int_{\Omega} u u_t \ln(1+u_t^2) \, dx.$$
We have, by Cauchy-Schwarz, Poincar\'{e}'s and Young's inequalities
$$\left| -2\beta \int_{\Omega} uv \, dx  \right| \leq \beta C_p \left[ ||\nabla u||_2^2+||\nabla v||_2^2\right],$$
which yields 
$$\psi'(t) \leq ||u_t||_2^2+||v_t||_2^2+(\beta C_p-1)||\nabla u||_2^2+(\beta C_p-1)||\nabla v||_2^2 - \alpha \int_{\Omega} u u_t \ln(1+u_t^2) \, dx.$$
We now deal with the last term. Define the following partition of $\Omega$.
$$\Omega^+ := \{x \in \Omega: |u_t(x)|> 1\} \quad \text{and} \quad \Omega^- := \{x \in \Omega: |u_t(x)| \leq 1\}.$$
Apply Young's and Poincar\'{e}'s inequalities to arrive at 

\begin{equation*}
\begin{split}
\int_{\Omega^{-}} u u_t \ln(1+u_t^2) \, dx &\leq \varepsilon C_p||\nabla u||_2^2 + C_{\varepsilon}\int_{\Omega^-} u_t^2 \ln^2(1+u_t^2) \, dx \\ 
&=\varepsilon C_p||\nabla u||_2^2 + C_{\varepsilon}\int_{\Omega^-} [u_t^2 \ln(1+u_t^2)] \, \ln(1+u_t^2) \, dx \\
&\leq \varepsilon C_p||\nabla u||_2^2 + C_{\varepsilon}\ln 2 \int_{\Omega^-} [u_t^2 \ln(1+u_t^2)]   \, dx \\
& \leq 2 \varepsilon C_p E_1(t) + C(\varepsilon)(-E_1'(t)).
\end{split}
\end{equation*}

On $\Omega^+$ we first apply H\"older's inequality and then \eqref{b2}, choose the exponents accordingly and also the Sobolev embedding to deduce

\begin{equation*}
\begin{split}
&\int_{\Omega^{+}} u u_t \ln(1+u_t^2) \, dx \\
& \qquad \leq \left(\int_{\Omega^+} |u|^{\gamma} \, dx \right)^{1/{\gamma}}\left( \int_{\Omega^+} (|u_t|\ln (1+u_t^2))^{{\gamma}/(\gamma-1)} \, dx \right)^{(\gamma-1)/{\gamma}}    \\ 
& \quad \quad  = \left(\int_{\Omega^+} |u|^{\gamma} \, dx \right)^{1/{\gamma}} \left( \int_{\Omega^+} |u_t|^{\frac{\gamma-2}{\gamma-1}}  (|u_t|^2\ln (1+u_t^2))^{1/(\gamma-1)} \ln (1+u_t^2) \, dx \right)^{(\gamma-1)/{\gamma}}\\
&\qquad \quad \leq C  ||u||_{\gamma} \left( \int_{\Omega^+} |u_t|^{\frac{\gamma-2}{\gamma-1}}  [(|u_t|^2)^{1+\rho}] \ln (1+u_t^2) \, dx \right)^{(\gamma-1)/{\gamma}}.
\end{split}
\end{equation*}
This implies
\begin{equation*}
\begin{split}
\int_{\Omega^{+}} u u_t \ln(1+u_t^2) \, dx &\leq C  ||u||_{\gamma} \left( \int_{\Omega^+} |u_t|^2 \ln (1+u_t^2) \, dx \right)^{(\gamma-1)/{\gamma}}\\
&  \qquad \leq C ||\nabla u||_2 \left(- E_1'(t)\right)^{(\gamma-1)/{\gamma}}\\
&   \qquad \qquad \leq C E_1^{1/2}(0) \left(- E_1'(t)\right)^{(\gamma-1)/{\gamma}}.
\end{split}
\end{equation*}
Here, we let $\gamma$ such that $2 < \gamma \leq \frac{2d}{d-2}$ if $d \geq 3$ and $\gamma > 2$ if $d=1, 2$. We also note that $\rho $ is chosen in such a way that $\rho = \frac{\gamma}{2}-1$.

Combination of the above estimates yields \eqref{d6}.

\end{proof}

\begin{lemma}\label{lemma-d7}
Let $(u,v)$ be the solution of \eqref{a1} and let
\begin{equation}\label{d8}
I(t) := \int_{\Omega} u_{tt} v_t \, dx -  \int_{\Omega} v_{tt} u_t \, dx. 
\end{equation}
Then, 
\begin{equation}\label{d9}
I'(t) \leq C\varepsilon E_1(t)+C\left(1+\frac{1}{2\varepsilon}\right)\left[-E_2'(t)\right]+C {[-E_1'(t)]}^{\frac{\gamma-2}{\gamma}} -\beta||v_t||_2^2+\alpha||u_t||_2^2. 
\end{equation}

\end{lemma}

\begin{proof}

A routine calculation yields
\begin{equation}\label{d10}
I'(t) = -\alpha \int_{\Omega} v_t u_{tt} \ln(1+u_t^2) \, dx - \alpha \int_{\Omega}\frac{2u_t^2v_tu_{tt}}{1+u_t^2}\, dx -\beta \int_{\Omega} v_t^2 \, dx + \beta \int_{\Omega} u_t^2 \, dx.
\end{equation}
We now estimate the first term on the right side of \eqref{d10}. H\"older's and Young's inequalities give
\begin{equation*}
\begin{split}
\left| -\alpha \int_{\Omega} v_t u_{tt} \ln(1+u_t^2) \, dx \right| &\leq  \alpha \int_{\Omega} \left| v_t u_{tt} \ln(1+u_t^2) \right| \, dx\\
& \leq \alpha \left[ \varepsilon \int_{\Omega} v_t^2 \ln(1+u_t^2) \, dx  + \frac{1}{4\varepsilon} \int_{\Omega} u_{tt}^2 \ln(1+u_t^2) \, dx\right].
\end{split}
\end{equation*}
Recalling \eqref{b5}, we see that the last term is bounded by $\frac{1}{4\varepsilon}(-E_2'(t))$. We now estimate the first term in the last line. Applying H\"older's inequality, the Sobolev embedding, \eqref{b5} and \eqref{b6} we deduce 
\begin{equation*}
\begin{split}
\int_{\Omega} v_t^2 \ln(1+u_t^2) \, dx &\leq \left( \int_{\Omega} |v_t|^{\gamma} \, dx\right)^{2/{\gamma}}\left( \int_{\Omega} {[\ln(1+u_t^2)]}^{\frac{\gamma}{(\gamma-2)}} \, dx\right)^{\frac{(\gamma - 2)}{\gamma}}\\
& \quad  \leq C_e ||\nabla v_t||_2^2\left( \int_{\Omega} {[\ln(1+u_t^2)]}^{\frac{\gamma}{(\gamma-2)}} \, dx\right)^{\frac{(\gamma - 2)}{\gamma}}\\
&\qquad \quad \leq 2C_eE_2(0) \left( \int_{\Omega} {[\ln(1+u_t^2)]}^{\frac{\gamma}{(\gamma-2)}} \, dx\right)^{\frac{(\gamma - 2)}{\gamma}}.\\
&\qquad \qquad \leq C \left( \int_{\Omega} {[\ln(1+u_t^2)]}^{\frac{\gamma}{(\gamma-2)}} \, dx\right)^{\frac{(\gamma - 2)}{\gamma}}.
\end{split}
\end{equation*}

Here, $\gamma$ is such that $2<\gamma \leq \frac{2d}{d-2}$ for $d \geq 3$, and $\gamma > 2$ if $d=1, 2$. We now estimate the last term in the above inequality as follows:\\
\begin{equation*}
\begin{split}
\int_{\Omega} \ln(1+u_t^2)^{\frac{\gamma}{(\gamma-2)}} \, dx &= \int_{\Omega} \ln(1+u_t^2) \ln(1+u_t^2)^{\frac{2}{(\gamma-2)}} \, dx \\
&= \int_{\Omega^{+}} \ln(1+u_t^2) \ln(1+u_t^2)^{\frac{2}{(\gamma-2)}} \, dx \\
  &\quad +\int_{\Omega^{-}} \ln(1+u_t^2) \ln(1+u_t^2)^{\frac{2}{(\gamma-2)}} \, dx. \\
\end{split}
\end{equation*}

For the integral over $\Omega^+$,  we apply \eqref{b2} with $\rho = \frac{\gamma-2}{2}> 0$, to get

\begin{equation*}
\begin{split}
 \int_{\Omega^{+}} \ln(1+u_t^2) \ln(1+u_t^2)^{\frac{2}{(\gamma-2)}} \, dx &\leq   \int_{\Omega^{+}} (C_{\rho}|u_t^2|^{\rho})^{\frac{2}{\gamma -2}}\ln(1+u_t^2)  \, dx \\
 & \quad = C_{\rho}^{\frac{2}{\gamma-2}} \int_{\Omega^{+}} |u_t^2|^{\frac{2\rho}{\gamma-2}} \ln(1+u_t^2)  \, dx \\
  & \qquad = C \int_{\Omega^{+}} u_t^2 \ln(1+u_t^2)  \, dx \\
  & \qquad \qquad  \leq C [-E_1'(t)].
\end{split}
\end{equation*}

Similarly, for the integral over $\Omega^-$,  we apply \eqref{b1} with $q = \frac{\gamma}{2}> 0$, to obtain

\begin{equation*}
\begin{split}
 \int_{\Omega^{-}} \ln(1+u_t^2) \ln(1+u_t^2)^{\frac{2}{(\gamma-2)}} \, dx &\leq   \int_{\Omega^{-}} C_{q}^{\frac{2}{\gamma-2}}   |u_t^2|^{2\frac{q-1}{\gamma-2}}      \ln(1+u_t^2)  \, dx \\
  & \qquad = C \int_{\Omega^{+}} u_t^2 \ln(1+u_t^2)  \, dx \\
  & \qquad \qquad  \leq C [-E_1'(t)].
\end{split}
\end{equation*}

Combining, all the latter estimates, we conclude

\begin{equation*}
\begin{split}
\int_{\Omega} v_t^2 \ln(1+u_t^2) \, dx & \leq C \left( \int_{\Omega} {[\ln(1+u_t^2)]}^{\frac{\gamma}{(\gamma-2)}} \, dx\right)^{\frac{(\gamma - 2)}{\gamma}}\\
&\qquad \qquad \leq C {\left[ -E_1'(t)\right]}^{\frac{(\gamma - 2)}{\gamma}},
\end{split}
\end{equation*}
hence,
\begin{equation*}
\begin{split}
\left| -\alpha \int_{\Omega} v_t u_{tt} \ln(1+u_t^2) \, dx \right| &\leq \alpha \varepsilon \int_{\Omega}v_t^2 \ln(1+u_t^2) \, dx 
+ \alpha \frac{1}{4\varepsilon} \int_{\Omega} u_{tt}^2 \ln(1+u_t^2) \, dx \\
& \qquad \leq C \varepsilon E_2(0) {[-E_1'(t)]}^{\frac{\gamma-2}{\gamma}} +\frac{1}{4\varepsilon}[-E_2'(t)].
\end{split}
\end{equation*}

 For the second term of \eqref{d10}, Cauchy-Schwarz and Young's inequalities imply that

\begin{equation*}
\begin{split}
\left| -2\alpha \int_{\Omega} \frac{u_t^2v_tu_{tt}}{1+u_t^2} \, dx\right|  &\leq  2\alpha \varepsilon \int_{\Omega} v_t^2 \, dx + 2\alpha \frac{1}{4\varepsilon}\int_{\Omega} \frac{u_t^2u_{tt}^2}{1+u_t^2} \frac{u_t^2}{1+u_t^2} \, dx \\
& \qquad \leq  4 \alpha \varepsilon E_1(t) + \frac{1}{2\varepsilon} \alpha \int_{\Omega} \frac{u_t^2u_{tt}^2}{1+u_t^2}  \, dx \\  & \qquad \qquad \leq 4 \alpha \varepsilon E_1(t) + \frac{1}{2\varepsilon} [-E_2'(t)].
\end{split}
\end{equation*}

Combination of  these estimates implies \eqref{d9} and this completes the proof.

\end{proof}

\begin{theorem}\label{thm:stability}
Assume \eqref{eq:cond-beta} holds and that  $u_0, v_0 \in H^2(\Omega) \cap H_0^1(\Omega)$  and $u_1, v_1 \in H_0^1(\Omega)$. Then, there exists a positive constant $C$ such that 

\begin{equation}\label{eq:energydecay}
E_1(t) \leq C\, {t^{-\frac{1}{(r+1)}}}, \quad \forall t > 0,
\end{equation}
 where $r = \max \{1, \frac{d-2}{2}\}$.

\end{theorem}

\begin{proof}
For $\eta, \eta_1> 0$ define the functional $L(t)$ as 

\begin{equation*}
L(t) := \eta \left[ E_1(t) + E_2(t) \right] + \eta_1 \psi(t) + I(t).
\end{equation*}

Using Lemma \ref{lemma-d4} and Lemma \ref{lemma-d7}, we get

\begin{equation*}
\begin{split}
L'(t) & = \eta E_1'(t) + \eta E_2'(t) + \eta_1 \psi' + I'(t) \\
       & \leq \eta E_1'(t) + \eta E_2'(t)  + C\varepsilon E_1(t) + C(1+\frac{1}{2\varepsilon})(-E_2'(t)) +C{(-E_1'(t))}^{\frac{(\gamma -2)}{\gamma}} \\
       &-\beta ||v_t||_2^2+\alpha||u_t||_2^2 + \eta_1 \Bigg[ ||u_t||_2^2+||v_t||_2^2 +(\beta C_p-1)||\nabla u||_2^2  \\
      &\quad    + (\beta C_p-1)||\nabla v||_2^2 + 2\varepsilon C_p E_1(t) + C(\varepsilon)(-E_1'(t)) + C {(-E_1'(t))}^{\frac{(\gamma -1)}{\gamma}} \Bigg]\\
     &= [\eta - C(\varepsilon)]E_1'(t) + [\eta - C(1+\frac{1}{2\varepsilon}) ] E_2'(t) \\
     &\quad + \varepsilon (2C_p+C)E_1(t) + (\eta_1 + \alpha)||u_t||_2^2 + (\eta_1-\beta)||v_t||_2^2\\
     & \qquad + \eta_1(\beta C_p-1)||\nabla u||_2^2 + \eta_1(\beta C_p -1)||\nabla v||_2^2 \\
     & \qquad \quad + C {(-E_1'(t))}^{\frac{(\gamma -1)}{\gamma}}  +C{(-E_1'(t))}^{\frac{(\gamma -2)}{\gamma}}.
\end{split}
\end{equation*}
Here, we choose $\eta_1 < \beta$ and recall that $\beta C_p < 1$. This yields
\begin{equation*}
\begin{split}
L'(t) & \leq [\eta - C(\varepsilon)]E_1'(t) + [\eta - C(1+\frac{1}{2\varepsilon})]E_2'(t) + \varepsilon(2C_p+C)E_1(t)\\
       & \quad + (\eta_1 + \alpha)||u_t||_2^2 - \alpha_1||v_t||_2^2 - \alpha_2||\nabla u||_2^2 -\alpha_2||\nabla v||_2^2 \\
              & \qquad +  C {(-E_1'(t))}^{\frac{(\gamma -1)}{\gamma}}  +C{(-E_1'(t))}^{\frac{(\gamma -2)}{\gamma}} - \alpha_3||u_t||_2^2+\alpha_3||u_t||_2^2 \\
        & \leq [\eta - C(\varepsilon)]E_1'(t) + [\eta - C(1+\frac{1}{2\varepsilon})]E_2'(t) \\
        &\qquad +[\varepsilon(2C_p+C)-\alpha_4] E_1(t) + (\eta_1+\alpha+\alpha_3)||u_t^2||_2^2\\
        &\qquad \qquad + C {(-E_1'(t))}^{\frac{(\gamma -1)}{\gamma}}  +C{(-E_1'(t))}^{\frac{(\gamma -2)}{\gamma}}.   
\end{split}
\end{equation*}

In above inequalities, we made use of  \eqref{d2} and the following side calculations. Taking $\delta = \min \{\alpha_1, \alpha_2, \alpha_3\}$ one has
\begin{equation*}
\begin{split}
&-\alpha_1||v_t^2||_2^2 -\alpha_2||\nabla u||_2^2 -\alpha_2||\nabla v||_2^2 - \alpha_3||u_t^2||_2^2\\
&\quad \leq -\delta \left[ ||v_t^2||_2^2 + \nabla u||_2^2 + ||\nabla v||_2^2 +||u_t^2||_2^2\right]\\
&\qquad \leq -\delta m E_1(t) = -\alpha_4 E_1(t).
\end{split}
\end{equation*}

Now, we choose $\varepsilon$ small so that 
$$ \varepsilon(2C_p+C)-\alpha_4 = -c_0 < 0,$$
and this implies that

\begin{equation*}
\begin{split}
L'(t) &\leq [\eta - \tilde{C}]E_1'(t) + [\eta - \tilde{\tilde{C}}]E_2'(t) -c_0 E_1(t) \\
       & \quad C||u_t^2||_2^2 + C {(-E_1'(t))}^{\frac{(\gamma -1)}{\gamma}}  +C{(-E_1'(t))}^{\frac{(\gamma -2)}{\gamma}}\\
       & \qquad + C(-E_1'(t))^{1/2} + C(-E_1'(t))\\
       & = [\eta - \tilde{C}-C]E_1'(t) + [\eta - \tilde{\tilde{C}}]E_2'(t) -c_0 E_1(t) \\
       & \quad + C(-E_1'(t))^{1/2} + C {(-E_1'(t))}^{\frac{(\gamma -1)}{\gamma}}  +C{(-E_1'(t))}^{\frac{(\gamma -2)}{\gamma}}.
\end{split}
\end{equation*}

We now choose $\eta$ so large that $\eta - \tilde{C}-C > 0$ and $\eta - \tilde{\tilde{C}} > 0$, so that  $L \geq 0$. Consequently, we arrive at
\begin{equation}\label{eq:derivativeL}
L'(t) \leq -c_0 E_1(t) + C(-E_1'(t))^{1/2} + C {(-E_1'(t))}^{\frac{(\gamma -1)}{\gamma}}  +C{(-E_1'(t))}^{\frac{(\gamma -2)}{\gamma}},
\end{equation}
where we remind that $\gamma > 2,$ if $d = 1, 2$ and $2 < \gamma \leq \frac{2d}{d-2}$, if $d \geq 3$.

For  $r \geq 1$, we multiply \eqref{eq:derivativeL} by $E_1^r$ and use Young's inequality to obtain for $\varepsilon_1 > 0$,
\begin{equation*}
\begin{split}
 E_1^r(t) L'(t) \leq -c_0 E_1^{r+1}(t) + \varepsilon_1 E_1^{2r}(t) + \varepsilon_1 E_1^{\gamma r}(t) + \varepsilon_1 E_1^{\frac{\gamma r}{2}} - \frac{C}{\varepsilon_1} E_1'(t).
\end{split}
\end{equation*}
This yields 

\begin{equation}\label{eq4-10}
\begin{split}
&E_1^r(t) L'(t) + \frac{C}{\varepsilon_1} E_1'(t) \\
 \quad \quad &\leq - \Big [ c_0 - \varepsilon_1 E_1^{r-1}(t) - \varepsilon_1 E_1^{r(\gamma -1)-1}(t) - \varepsilon_1 E_1^{r(\frac{\gamma}{2}-1)-1}(t) \Big] E_1^{r+1}(t).
\end{split}
\end{equation}
As $r \geq 1$ and $\gamma > 2,$ clearly $r(\gamma-1)-1> 0$. We need to choose $r \geq 1$ such that $r \geq \frac{2}{\gamma - 2}$. At this point, we distinguish two situations

\begin{itemize}

\item If $d = 1, 2$ we take $\gamma  = 4$; hence, we get $r=1$. 

\item If $d \geq 3$, we take $\gamma = \frac{2d}{d-2}$; hence $r \geq \frac{d-2}{2}$.

\end{itemize}

Consequently, \eqref{eq4-10} becomes

\begin{equation*}
E_1^r(t) + \frac{C}{\varepsilon_1} E_1'(t) \leq - \Big [c_0 - \varepsilon_1 E_1^{r-1}(0) - \varepsilon_1 E_1^{r(\gamma -1)-1}(0) -\varepsilon_1 E_1^{r(\frac{\gamma}{2}-1)}(0) \Big ] E_1^{r+1}(t).
\end{equation*}
We, then, choose $\varepsilon_1$ so small that 

 $$ A := c_0 - \varepsilon_1 E_1^{r-1}(0) - \varepsilon_1 E_1^{r(\gamma -1)-1}(0) -\varepsilon_1 E_1^{r(\frac{\gamma}{2}-1)}(0)  > 0.$$
Therefore, \eqref{eq4-10} gives
$$ CE_1'(t) + E_1(t)^r L'(t) \leq - A E_1^{r+1}(t).$$

Now, observe that 
\begin{equation*} 
 \begin{split}
 \left( E_1^r(t) L(t) + C E_1(t)\right)' &= rE_1^{r-1}(t)E_1'(t)L(t) + E_1^r(t)L'(t) + C E_1'(t)  \\
 & \quad \leq E_1^r(t)L'(t) + CE_1'(t)  \leq - A E_1^{r+1}(t),
 \end{split}
 \end{equation*}
and this implies that
\begin{equation}\label{eq-4.10}
E_1^{r+1}(t) \leq -\frac{1}{A} \Big[ E_1^r(t)L(t) + CE_1(t) \Big]'.
\end{equation}
Integration of \eqref{eq-4.10} yields
$$ \int_0^t E_1^{r+1}(s) \, ds \leq -\frac{1}{A} \int_0^t [E_1^r L + C E_1]'(s) \, ds.$$
Note that
$$ \int_0^t E_1^{r+1}(t) \, ds \leq \int_0^t E_1^{r+1}(s) \, ds, \quad \text{for} \quad s \in [0, t],$$
which implies
\begin{equation*}
\begin{split}
E_1^{r+1}(t) t \leq \int_0^t E_1^{r+1}(s) \, ds &\leq -\frac{1}{A}[(E_1^r L + C E_1)(t) - (E_1^r L + C E_1)(0)] \\
&\leq \frac{1}{A} \left(E_1^r L + C E_1\right)(0).
\end{split}
\end{equation*}
Hence,
$$ E_1^{r+1}(t) \leq \frac{C}{t}.$$
This implies \eqref{eq:energydecay} and completes the proof.

\end{proof}

\def\ocirc#1{\ifmmode\setbox0=\hbox{$#1$}\dimen0=\ht0 \advance\dimen0
  by1pt\rlap{\hbox to\wd0{\hss\raise\dimen0
  \hbox{\hskip.2em$\scriptscriptstyle\circ$}\hss}}#1\else {\accent"17 #1}\fi}

\noindent \textbf{Authors' contributions}\\
All the authors participated equally in the paper.\\

\noindent \textbf{Declaration of competing interest}\\
There is no conflict of Interest.\\

\noindent \textbf{Data availability}\\
There is no Data available.

\end{document}